\documentclass[a4paper,12pt]{article}

\usepackage{amsfonts,amssymb,amsmath,amsthm}
\usepackage{bbm}
\usepackage{xcolor}
\usepackage{graphics}
\usepackage{graphicx}
\usepackage{mathtools}
\usepackage{hyperref} 
\usepackage{ulem} 

\numberwithin{equation}{section}

\def\un{\mathbbm{1}}

\def\E{\mathbb{E}}
\def\P{\mathbb{P}}
\def\R{\mathbb{R}}

\def\N{\mathbb{N}}

\def\d{\partial}

\def\t{-}
\def\T{_}
\def\eq{=}
\def\pl{+}
\newtheorem{theorem}{Theorem}[section]
\newtheorem{lemme}[theorem]{Lemma}
\newtheorem{corollary}[theorem]{Corollary}

\newtheorem{proposition}[theorem]{Proposition}

\theoremstyle{remark}
\newtheorem{remark}{Remark}
\newtheorem{example}{Example}

    \def\restriction#1#2{\mathchoice
                  {\setbox1\hbox{${\displaystyle #1}_{\scriptstyle #2}$}
                  \restrictionaux{#1}{#2}}
                  {\setbox1\hbox{${\textstyle #1}_{\scriptstyle #2}$}
                  \restrictionaux{#1}{#2}}
                  {\setbox1\hbox{${\scriptstyle #1}_{\scriptscriptstyle #2}$}
                  \restrictionaux{#1}{#2}}
                  {\setbox1\hbox{${\scriptscriptstyle #1}_{\scriptscriptstyle #2}$}
                  \restrictionaux{#1}{#2}}}
    \def\restrictionaux#1#2{{#1\,\smash{\vrule height .8\ht1 depth .85\dp1}}_{\,#2}}

\definecolor{darkgreen}{rgb}{0,0.3,0}

\def\be{\begin{eqnarray}}
\def\ee{\end{eqnarray}}
\def\ben{\begin{eqnarray*}}
\def\een{\end{eqnarray*}}

\def\me{\medskip\noindent}
\def\bi{\bigskip\noindent}
\newcommand{\ba}{\vspace{0mm}\begin{equation*}\begin{aligned}}
\newcommand{\ea}{\vspace{0mm}\end{aligned}\end{equation*}}
\newcommand{\ban}{\vspace{0mm}\begin{equation}\begin{aligned}}
\newcommand{\ean}{\vspace{0mm}\end{aligned}\end{equation}}

\title{\center Perpetual integrals convergence and extinctions in population dynamics}

\author{Camille Coron$^{\textup{(a)}}$,  Sylvie M\'el\'eard$^{\textup{(b)}}$, Denis Villemonais$^{\textup{(c)}}$,\\
{\small $^{\textup{(a)}}$ Laboratoire de Math\'ematiques d'Orsay, Univ. Paris-Sud, CNRS, }\\
{\small  Universit\'e Paris-Saclay, 91405 Orsay, France}\\
{\small $^{\textup{(b)}}$ Centre de Math\'ematiques Appliqu\'ees, Ecole polytechnique, CNRS,}\\
{\small  Universit\'e Paris-Saclay, 91128 Palaiseau Cedex, France}\\
{\small $^{\textup{ (c )}}$ IECL, UMR 7502, CNRS,}\\
{\small Universit\'e de Lorraine, Vand{\oe}uvre-l\`es-Nancy, France}\\
{\small Inria, TOSCA team, Villers-l\`es-Nancy, France}
}

\date{\today}

\begin{document}

\maketitle

\begin{abstract}
In this article we use a criterion for the integrability   of paths of one-dimensional diffusion processes from which we derive new insights on allelic fixation in several  situations. This well known criterion involves a simple necessary and sufficient condition based on  scale function and speed measure. We provide a new simple proof for this result and also obtain explicit bounds for the moments of such integrals. We also extend this criterion to non-homogeneous processes by use of Girsanov's transform. We apply our results to multi-type population dynamics: using the criterion with appropriate time changes, we characterize the behavior of proportions of each type  before population extinction in different situations.  \end{abstract}


\section{Introduction}
Our  motivations  in this paper come from population genetics. The first question  concerns  the dynamics of an allelic proportion in a variable size population  going to extinction. We wonder whether the allelic proportion will attain $1$  or $0$ (fixation or loss  of the allele in the population) before population extinction.
The second question concerns the dynamics of the respective proportions of $L$ neutral alleles until fixation of one of them. We are asking about simultaneous allele extinctions  or not. 

\smallskip \noindent In  both cases,   we need to slow down the dynamics before either population extinction or allele fixation by the use of  a time change. That  leads us to study quantities of the form $\int_{0}^{T_{0}} f(Z_{s})ds$ (which are referred to as perpetual integrals~\cite{SalminenYor2005}),  for a nonnegative  diffusion process $Z$ and  $T_{0}$ the hitting time of $0$, or $\int_{0}^{T_{0}\wedge T_{1}} f(X_{s})ds$, for a   diffusion process $X\in [0,1]$ and  $T_{0},T_{1}$ the hitting times of $0$ and $1$. We need to know whether such integrals are finite or not. In Section 2, we state and prove a general criterion involving a necessary and sufficient condition based on the scale function and speed measure  of the nonnegative diffusion process $Z$, which ensures that the integral $\int_{0}^{T_{0}} f(Z_{s})ds$ is finite almost surely or infinite almost surely. This 0-1 law criterion was already known and proved using a combination of the local time formula, the Ray-Knight Theorem and Jeulin's Lemma (see~\cite{EngelbertTittel2002,MijatovicUrusov2012}). We provide a simpler proof which also provides explicit bounds for the moments of perpetual integrals and can be easily extended to more general one dimensional Markov processes. Then, we extend this result to a diffusion taking values in a compact subset and finally to non-homogeneous processes by the use of Girsanov's transform. Applications to standard population  models are given. 

\smallskip \noindent  In Section 3, we apply these  general integrability results  to  several open  allele fixation problems. The use   of tricky time changes dramatically simplifies these questions. 
Subsection 3.1 concerns the dynamics of an allele proportion in a population with variable size.
 We assume that the total  population size goes to $0$ almost surely. We give a necessary and sufficient condition for the coupled  logistic population size dynamics and allelic neutral Wright-Fisher equation (with variable size) to get allelic fixation before extinction almost surely. The condition is satisfied when the population size dynamics is a logistic Feller stochastic differential equation. Nevertheless, we give examples of  population size dynamics for which extinction  occurs before fixation with positive probability,  emphasizing by this way  the necessity of taking into account the behavior of the population size, in particular near extinction. We also study a case with allelic selection using a Girsanov's transform 
for the coupled system of population size and allelic proportion. 
In Subsection~3.2,  we consider a neutral $L$-type Wright-Fisher diffusion. We  show that one of the alleles is fixed almost surely in finite time and that until that time, the population experiences successive (and not simultaneous)  allele extinctions. These results are proved by induction on $L$ and using a time change based on the fixation time: we slow down  time before fixation to observe the successive allele extinctions. 

\section{Integrability properties for diffusion processes}

\subsection{General diffusion processes on $[0,+\infty)$}
Let us consider a general one-dimensional diffusion process $\,(Z_{t}, t\geq 0)\,$ (that is a continuous strong Markov process) with values in $[0,+\infty)$. 
We denote by $T_{z}$ the hitting time of $z\in [0,+\infty)$  by the process $Z$:
$$T_{z} = \inf\{t\geq 0, Z_{t} = z\},$$ if the set is non empty and $T_{z} $ is infinite otherwise. 
When the process $Z$ has to be specified, this time will be denoted $T_{z}^Z$.

\me Let us denote by  $\mathbb{P}_z$ the law of $Z$ starting from $z$. 
We assume that $Z$ is regular ($\forall z \in (0,+\infty), \forall y \in (0,+\infty)$, $\mathbb{P}_{z}(T_{y}<+\infty)>0$). This implies (see 
Revuz-Yor \cite[VII-Proposition~3.2]{Revuz}) that for any $a<b\in (0,+\infty)$ and $a\leq z \leq b$,
$$\mathbb{E}_{z}(T_{a}\wedge T_{b})<+\infty.$$

 \me  We may associate with the process a  scale function $\,s\,$  and   its locally finite speed measure $m$ on $[0,+\infty)$ (see \cite[Chapter VII]{Revuz}). We will assume moreover that,  for all $z\in(0,+\infty)$, 
\be
\label{extinction}
\mathbb{P}_z(T_0 = T_{0}\wedge T_{e}<+\infty)=1,\ee
 where $T_{e}$ is the explosion time.

\smallskip
\begin{lemme}
\label{lemextinction}
Condition \eqref{extinction} is equivalent to \be
\label{extinctionbis} \,s(+\infty)=+ \infty\quad ; \quad s(0)>  - \infty \quad  ; \quad \int_{0+} (s(y)-s(0))\,m(dy)<+\infty.\ee
\end{lemme}


\me Note that  Condition \eqref{extinctionbis} is well known  in the case where $Z$ is solution of a stochastic differential equation (cf. \cite{Karatzas} p.348, \cite{Ikeda} p.450).

\begin{proof}
Assume first that \eqref{extinction} is satisfied. As $Z$ has scale $\,s$, $\,s(Z)\,$   is a local martingale on  $(s(0), s(+\infty))$ such that $T_{s(0)}^{s(Z)}<T_{s(+\infty)}^{s(Z)}$ a.s.. We deduce that  $s(0)>  - \infty$ and $s(+\infty)=+\infty$.    The diffusion $s(Z)$ has a natural scale with speed measure $\tilde m = m\circ s^{-1}$ (see \cite{Revuz}, Chapter VII). Since it  attains $s(0)$ in finite time almost surely, we deduce using 
 \cite[Theorem~51-2]{Rogers-Williams}  that  $\int_{s(0)+} (u-s(0))\,\tilde m(du)<+\infty$. As
 $\int_{s(0)+} (u-s(0))\,\tilde m(du)<+\infty \Longleftrightarrow \int_{0+} (s(y)-s(0))\,m(dy)<+\infty$, we obtain  \eqref{extinctionbis}.

\noindent  Conversely, assume \eqref{extinctionbis}. Conditions $s(0)>  - \infty$ and  $s(+\infty)=\pl \infty$ imply that the local martingale $s(Z)$  doesn't explode  a.s.. Since $\int_{0+} (s(y)-s(0))\,m(dy)<+\infty$, then $ \int_{s(0)+} (u-s(0))\,\tilde m(du)<+\infty$ and the process $s(Z)$ attains $s(0)$ in finite time a.s.,   so does  the process $Z$. That concludes the proof. 
\end{proof}

\bi Since the function $s$ is defined up to a constant, we choose by convention  $s(0)=0$ as soon as $s(0)>-\infty$.

\bi 
In the following theorem, we prove a $0-1$ law for the  finiteness/infiniteness  of perpetual integrals of diffusion processes and provide explicit bounds for their moments. This $0-1$ law has been extensively studied and already proved by different ways in the literature. Its first proof goes back to~\cite{EngelbertTittel2002} (see also~\cite{MijatovicUrusov2012}) using a combination of the local time formula, Ray-Knight Theorem and Jeulin's Lemma. Attempts to simplify this approach are provided in~\cite{Cui2014} for some stochastic differential equations using an appropriate space change. The almost sure finiteness criterion has also been recovered by simple means in~\cite{Khoshnevisan2006}, where the existence of a non-explicit exponential moment for perpetual integrals is also proved. Proofs of the 0-1 law in particular settings are given in \cite{Engelbert-Senf, Foucart-Henard}. In \cite{Salminen-Yor2}, the authors define perpetual integrals as the first hitting times of diffusion processes, and illustrate how the Laplace transform of some perpetual integrals can be found using Feynman-Kac formula.

\bi

\me \begin{theorem}
\label{Integ-Ext} Let $\,(Z_{t}, t\geq 0)\,$ be a regular diffusion process on $\,[0,+\infty)$ with  scale function $s$ and speed measure $m$ on $(0,+\infty)$ satisfying \eqref{extinctionbis}.
Let also $f$ be a non-negative measurable function on $(0,+\infty)$ which is locally integrable on $(0,+\infty)$.  Then, for all $z>0$ and all $n\geq 1$,
\begin{align*}
\E_z\left[\left(\int_{0}^{T_{0}} f(Z_{s})\, ds\right)^n\,\right]\leq n!\,\left(\int_0^\infty {s(y) \,f(y)} \,m(dy)\right)^n
\end{align*}
and
\begin{align*}
\int_{0^+} {s(y) \,f(y)} \,m(dy) < + \infty\ 
&\Longleftrightarrow\  \int_{0}^{T_{0}} f(Z_{s}) \,ds < + \infty \quad \P_z-\hbox{almost surely}\\
\int_{0^+} {s(y) \,f(y)} \,m(dy) = + \infty\ 
&\Longleftrightarrow\  \int_{0}^{T_{0}} f(Z_{s})\, ds = + \infty \quad \P_z-\hbox{almost surely}.
\end{align*}
\end{theorem}

\me \begin{proof} Because of the non-explosion assumption~\eqref{extinctionbis}, we have $\int_0^{T_0} f(Z_s)\,ds<\infty \Leftrightarrow \forall k\in\N,$ $\int_0^{T_0} f(Z_s)\mathbf{1}_{Z_s\leq k}\,ds<\infty$ and $\int_0^{T_0} f(Z_s)\,ds=\infty \Leftrightarrow \exists k\in\N$ such that  $\int_0^{T_0} f(Z_s)\mathbf{1}_{Z_s\leq k}\,ds=\infty$. Hence it is sufficient to prove Theorem~\ref{Integ-Ext-bounded} for functions $f$ satisfying $\int_a^\infty f(x)\,s(x)\,m(dx)<\infty$ for all $a>0$. We make this assumption from now on.

As $Z$ has  scale function  $s$ and speed measure $m$, the process $s(Z)$ is on a natural scale with speed measure $m\circ s^{-1}$.  Then it is enough to prove the result if the process $\,Z\,$ is on a natural scale ($s=identity$); the general case will follow immediately. In particular,  we have the following Green formula (see~[Chapter~23] of \cite{Kallenberg})

\begin{align*}
\E_x\Big(\int_0^{T_0} f(Z_s)\,ds\Big)&= \int_{(0,+\infty)} 2\,(x\wedge y)\,f(y)\,m(dy).
\end{align*}
One easily checks that, under $\P_x$ for any $x\in (0,+\infty)$, $\int_0^{T_0} f(Z_s)\,ds<+\infty$ satisfies a $0-1$ law. Indeed, we have
\begin{align*}
\int_0^{T_0} f(Z_s)\,ds=\sum_{k=1}^\infty \int_{T_{x/k}}^{T_{x/(k+1)}} f(Z_s)\,ds,
\end{align*}
where the $\int_{T_{x/k}}^{T_{x/(k+1)}} f(Z_s)\,ds$, $k\geq 1$, are non-negative independent (because of the strong Markov property) random variables which are almost surely finite (in fact with finite expectation, because of our assumptions and the Green's formula applied under $\P_{x/k}$ up to time $T_{x/{k+1}}$). Hence the above series is finite with probability zero or one.

\medskip
Assume first that $\int_{(0,+\infty)} y\,f(y)\,m(dy)<+\infty$. Then $\int_0^{T_0}f(Z_s)ds<\infty$ almost surely and, for all $n\geq 1$,
\begin{align*}
\E_x\left[\left(\int_0^{T_0}f(Z_s)ds\right)^n\right]
&=\E_x\left[n\int_{0}^{T_0}f(Z_s)\left(\int_{s}^{T_0}f(Z_u)\,du\right)^{n-1}ds\right]\\
&=n\,\int_0^\infty\E_x\left[\mathbf{1}_{s<T_0} f(Z_s)\left(\int_{s}^{T_0}f(Z_u)\,du\right)^{n-1}\right]\,ds\\
&=n\,\E_x\left[\int_0^{T_0}f(Z_s)\E_{Z_s}\left(\left(\int_0^{T_0}f(Z_u)du\right)^{n-1}\right)ds\right],
\end{align*}
where we used the Markov property. We immediately deduce by induction that
\begin{align*}
\E_x\left[\left(\int_0^{T_0}f(Z_s)ds\right)^n\right]\leq n! \left(\int_{(0,+\infty)}\,2yf(y)m(dy)\right)^n.
\end{align*}
This concludes the proof of the first part of Theorem~\ref{Integ-Ext} (the inequality is trivial when $\int_{(0,+\infty)} y\,f(y)\,m(dy)=+\infty$). 

\medskip
Assume now that $\int_{(0,+\infty)} y\,f(y)\,m(dy)=+\infty$ and fix $x\in(0,+\infty)$. For all $k\geq 1$, we set
\begin{align*}
f_k(y)=\begin{cases}
f(y)&\text{if }y\geq 1\\
f(y)\wedge k&\text{if }y<1.
\end{cases}
\end{align*}
In particular, $\int_{(0,+\infty)} f_k(y)\,y\,m(dy)<\infty$ for all $k\geq 1$ and hence, using the inequalities established above and then the fact that $\int_{(0,+\infty)}\,2yf_k(y)\,m(dy)$ goes to infinity and the fact that $yf(y)m(dy)$ is assumed to be finite on neighborhood of $+\infty$, we deduce that
\begin{align*}
\E_x&\left[\left(\int_0^{T_0}f_k(Z_s)ds\right)^2\right]\leq 2 \left(\int_{(0,+\infty)}\,2y\,f_k(y)\,m(dy)\right)^2\\
&\quad\quad\begin{aligned}
&\leq 2\left(\int_{(0,+\infty)}2\,(y\wedge x) f_k(y)\,m(dy)+\int_x^{\infty} 2(y-x)f(y)\,m(dy)\right)^2\\
&\leq 4\,\left(\int_{(0,+\infty)}2\,(y\wedge x) f_k(y)\,m(dy)\right)^2+4\left(\int_x^{\infty} 2(y-x)f(y)\,m(dy)\right)^2\\
&\leq 5\,\left(\int_{(0,+\infty)}2\,(y\wedge x) f_k(y)\,m(dy)\right)^2\quad\text{ for $k$ large enough}\\
&\leq 5\,\left[\E_x\left(\int_0^{T_0}f_k(Z_s)ds\right)\right]^2 \quad\text{ for $k$ large enough}.
\end{aligned}
\end{align*}
We deduce that, for $k$ large enough,
\begin{align*}
\P_x\Big(\int_0^{T_0}f_k(Z_s)ds\geq \frac{\E_x\left(\int_0^{T_0}f_k(Z_s)ds\right)}{2}\Big)\geq\frac{1}{20}.
\end{align*}
Indeed, for any random variable $Y\geq 0$ such that $\E(Y^2)\leq 5\E(Y)^2$, we have, setting $M=\E(Y)$,
\begin{align*}
5M^2&\geq \E(Y^2)\geq \E(Y^2\mid Y\geq M/2)\P(Y\geq M/2)\\
&\geq  \E(Y\mid Y\geq M/2)^2\,\P(Y\geq M/2)\\
&=\frac{\E(Y\mathbf{1}_{Y\geq M/2})^2}{\P(Y\geq M/2)}\geq \frac{M^2/4}{\P(Y\geq M/2)}
\end{align*}
and hence $\P(Y\geq M/2)\geq 1/20$.
Now using the fact that $f_k$ is increasing in $k$, we deduce that, for $k$ large enough,
\begin{align*}
\P_x\Big(\int_0^{T_0}f(Z_s)ds\geq \frac{\E_x\left(\int_0^{T_0}f_k(Z_s)ds\right)}{2}\Big)\geq 1/20.
\end{align*}
Since $\E_x\left(\int_0^{T_0}f_k(Z_s)ds\right)$ is not bounded in $k$, we deduce that
\begin{align*}
\P_x\Big(\int_0^{T_0}f(Z_s)ds=+\infty\Big)\geq 1/20.
\end{align*}
This and the fact that $\{\int_0^{T_0}f(Z_s)ds=+\infty\}$ satisfies a $0-1$ law conclude the proof of Theorem~\ref{Integ-Ext}.
\end{proof}

\bi The equivalences stated  in Theorem \ref{Integ-Ext} are particularly  useful when $\,Z\,$ is solution of a  one\t dimensional stochastic differential equation
\be
\label{SDE}dZ\T t \eq  \sigma(Z\T t ) d B\T t \pl b(Z\T t ) dt\quad;\quad Z\T 0 >0,\ee 
where $B$ is a one dimensional Brownian motion, and  $\sigma:(0,+\infty)\rightarrow (0,+\infty)$ and $b:(0,+\infty)\rightarrow\R$ are measurable functions such that $b/\sigma^2$ is locally integrable and such that $m(]a,b[)\in(0,+\infty)$ for all $0<a<b<+\infty$. Here $s$ and $m$ are the scale function (up to a constant) and speed measure equal to
\be
\label{scalespeed}s(x) = \int_{c}^x
\exp\Big(-2 \int_{c}^y\frac{b(z)}{\sigma^2(z)} dz \Big)dy \quad ; \quad m(dx) = \frac{2 dx}{s'(x)\,\sigma^2(x)},\ee
as detailed in Chapter~23 of~\cite{Kallenberg}. In particular, $Z$ is a regular diffusion.

\me 
\begin{corollary}
\label{cor:to}
Assume that $Z$ is solution of \eqref{SDE} with $\,s(+\infty)=+\infty\,$ and $\,\int_{0+} s(y)\,m(dy)<+\infty$.
Let us consider a non negative locally bounded measurable function $f$ on $(0,+\infty)$. Then, under $\mathbb{P}_z$,
\begin{align*}
\int_{0^+} \frac{f(y)s(y)}{s'(y)\sigma^2(y)} \,dy = + \infty\quad
&\Longleftrightarrow\quad \int_{0}^{T_{0}}f(Z_{s}) \,ds = + \infty \quad \hbox{almost surely},\\
\int_{0^+} \frac{f(y)s(y)}{s'(y)\sigma^2(y)} \,dy < + \infty\quad
&\Longleftrightarrow\quad \int_{0}^{T_{0}} f(Z_{s}) \, ds < + \infty \quad \hbox{almost surely}.
\end{align*}
 \end{corollary}

\bi Let us give two examples for  population size processes.

\bi \begin{example} Branching process with immigration. Let us consider  the  solution of the stochastic differential equation
$$dN_{t} = \sigma \sqrt{N_{t}} dB_{t} + \beta dt, \ \beta>0.$$
  The scale function   $s(x) = \int_{1}^x \exp\Big(-\int_{1}^y {\frac{2\beta}{\sigma^2 z}} dz\Big) dy =  {\frac{\sigma^2}{\sigma^2-{2\beta}}} \big(x^{1-{\frac{2\beta}{\sigma^2}}}-1\big)$ and $m(dx) = \frac{2 x^{2\beta/\sigma^2}dx}{\sigma^2 x}\,$ except when $\beta/ \sigma^2 \eq 1/2$ for which $s(x)\eq \log x $ and $m(dx) \eq \frac{ 2 dx}{ \sigma^2}$, cf. \eqref{scalespeed}.
Then $$\eqref{extinction} \ \Longleftrightarrow  \ \beta/ \sigma^2 < 1/2.$$
\me Applying Corollary \ref{cor:to}  with $\,f(y) = 1/y^\alpha$, we obtain
\be
\label{int}
\int_{0}^{T_{0}} \frac{1}{(N_{s})^\alpha} \,ds &=& + \infty\quad  a.s. \ \Longleftrightarrow \alpha\geq 1\nonumber;\\
\int_{0}^{T_{0}} \frac{1}{ (N_{s})^\alpha} \,ds &<& + \infty\quad  a.s. \ \Longleftrightarrow \alpha< 1
\ee
since  
$ \int_{0^+} \frac{1}{y^\alpha} \,dy = + \infty\Longleftrightarrow \alpha\geq 1$.
In the particular case $\alpha=1$, the authors of~\cite{Foucart-Henard} propose an other approach based on self-similarity properties.
\end{example}
\bi \begin{example} Logistic diffusion process. Let us consider the process 
$$dN_t=\sqrt{N_t}\,dB_t  + N_t\,(b - c\,N_t)\,dt \ ;\ N_0>0, $$ where  $\,b, c>0$. 
Then $\, s(y) = \int_0^y e^{c z^2  -  2 bz } dz\,$ and $\,m(dy) = \displaystyle\frac{2 e^{ - c y^2   +  2 by }}{y} dy$ and $\,\int_{0^+} s(y) m(dy) <+\infty$, since $\,\frac{s(y)}{s'(y)\,y} \to\T {y\to 0} 1$. (Note that if $\,c= 0$, the condition $s(+\infty) = +\infty$ is not satisfied). It is immediate to check that   \eqref{int} also holds.
\end{example}

\subsection{General diffusion processes on $(a,b)$}

Let us consider a general diffusion process $\,(X_{t}, t\geq 0)\,$ with   scale function $\,s\,$ and  locally finite speed measure $m$ on $(a,b)$, with $-\infty<a<b<+\infty$. Let us denote by $T_{a}$ and $T_b$ the hitting times of $a$ and $b$ respectively by the process $X$. We assume that, for all $x\in(a,b)$, $\mathbb{P}_x(T_a\wedge T_b<+\infty)=1$. This is the case if and only if one of the following properties is satisfied

\smallskip \noindent $(i) -\infty<s(a)<s(b)<+\infty$ ; $\int_{a^+} (s(y)-s(a))\,m(dy)<+\infty$ and $\int^{b^-}(s(b)-s(y)\,m(dy)<+\infty$;

\me $(ii) -\infty<s(a)$ and $s(b)=+\infty$ ;  $\int_{a^+} (s(y)-s(a))\,m(dy)<+\infty$;

\smallskip \noindent $(iii) \,s(a)=-\infty$ and $s(b)<+\infty$ ;  $\int^{b^-}(s(b)-s(y))\,m(dy)<+\infty$.

\bi
\begin{theorem}
\label{Integ-Ext-bounded}
Fix $x\in(a,b)$ and let $f:(a,b)\rightarrow \R_+$ be a locally bounded measurable function. Then \begin{align*}
&\int_{a^+} {(s(y)-s(a)) \,f(y)} m(dy) = + \infty\  \Longleftrightarrow\\
&\hskip 3cm  \P_x\Big(\{\int_{0}^{T_{a}} f(X_{s}) ds = + \infty\}\cap \{T_a<T_b\}\Big) =  \P_x\big(T_a<T_b\big)\\   
&\int_{a^+} {(s(y)-s(a)) \,f(y)} m(dy) < + \infty\  \Longleftrightarrow\\
&\hskip 3cm \ \P_x\Big(\{\int_{0}^{T_{a}} f(X_{s}) ds < + \infty\}\cap \{T_a<T_b\}\Big) =  \P_x\big(T_a<T_b\big) 
\end{align*}
A similar result holds at the boundary $b$:
\begin{align*}
&\int^{b^-} {(s(b)-s(y)) \,f(y)} m(dy) = + \infty\  \Longleftrightarrow\\
&\hskip 3cm \P_x\Big(\{\int_{0}^{T_{b}} f(X_{s}) ds  = + \infty\}\cap \{T_b<T_a\}\Big) =  \P_x\big(T_b<T_a\big) \\       
&\int^{b^-} {(s(b)-s(y)) \,f(y)} m(dy) < + \infty \
\Longleftrightarrow\\
&\hskip 3cm \P_x\Big(\{\int_{0}^{T_{b}} f(X_{s}) ds < + \infty\}\cap \{T_b<T_a\}\Big) =  \P_x\big(T_b<T_a\big).
\end{align*}
\end{theorem}

\me \begin{proof}
As in the proof of Theorem \ref{Integ-Ext}, it is enough to prove the result in the case where  $\,s\,$  is the identity function. 

\smallskip \noindent 
Without loss of generality, we take $(a,b)=(0,1)$. Let us consider $x\in (0,1)$, fix $\varepsilon\in(0,1-x)$ and consider a locally finite measure $m^{\varepsilon}$ on $(0,+\infty)$ such that $\restriction{m^\varepsilon}{(0,1-\varepsilon)}=\restriction{m}{(0,1-\varepsilon)}$. Let $X^\varepsilon$ be a diffusion process on natural scale on $(0,+\infty)$ with speed measure  $m^\varepsilon$ and starting from $x$, built as a time change of the same Brownian motion as $X$.
Because of this construction, $X$ and $X^\varepsilon$ coincide up to time $T_0$ on the event $\{T_0<T_{1-\varepsilon}\}$. 

Now, by Theorem~\ref{Integ-Ext} applied to $X^\varepsilon$ and $f^{\varepsilon}:y\mapsto f(y)\un_{y\leq 1-\varepsilon}$, we deduce that
\begin{align*}
 \int_{0}^{T_{0}} f(X^\varepsilon_{s})\un_{X^\varepsilon_s\leq 1-\varepsilon} \,ds = + \infty \quad\hbox{almost surely}\ &\Longleftrightarrow\ \int_{0^+} {y \,f(y)} m(dy) = + \infty,\\
\int_{0}^{T_{0}} f(X^\varepsilon_{s})\un_{X^\varepsilon_s\leq 1-\varepsilon}\, ds < + \infty \quad \hbox{almost surely}\ &\Longleftrightarrow\ \int_{0^+} {y \,f(y)} m(dy) < + \infty.
\end{align*}
Since $X$ and $X^\varepsilon$ coincide up to time $T_0$ on the event $T_0<T_{1-\varepsilon}$, we deduce that, up to negligible events,
\begin{align*}
\int_{0^+} {y \,f(y)} m(dy) = + \infty\quad&\Longrightarrow \quad  \int_{0}^{T_{0}} f(X_{s})\un_{X_s\leq 1-\varepsilon} \,ds = + \infty\ \hbox{on $T_0<T_{1-\varepsilon}$.}\\
\int_{0^+} {y \,f(y)} m(dy) < + \infty\quad&\Longrightarrow \quad  \int_{0}^{T_{0}} f(X_{s})\un_{X_s\leq 1-\varepsilon}\, ds < + \infty\ \hbox{on $T_0<T_{1-\varepsilon}$.}
\end{align*}
But  on  $T_0<T_{1-\varepsilon}$,  $X_s\leq 1-\varepsilon$ holds for $s\leq T_0$, so that, up to $\P_{x}$-negligible events,
\begin{align*}
\int_{0^+} {y \,f(y)} m(dy) = + \infty\quad&\Longrightarrow \quad  \int_{0}^{T_{0}} f(X_{s})ds = + \infty\quad\hbox{on $T_0<T_{1-\varepsilon}$.}\\
\int_{0^+} {y \,f(y)} m(dy) < + \infty\quad&\Longrightarrow \quad  \int_{0}^{T_{0}} f(X_{s}) ds < + \infty\quad\hbox{on $T_0<T_{1-\varepsilon}$.}
\end{align*}
The continuity of the paths of $X$ implies that
$$\{T_0<T_1\}=\cup_{0<\varepsilon < 1-x} \{T_0<T_{1-\varepsilon}\},$$
which yields, up to negligible events,
\begin{align*}
\int_{0^+} {y \,f(y)} m(dy) = + \infty\quad&\Longrightarrow \quad  \int_{0}^{T_{0}} f(X_{s})ds = + \infty\quad\hbox{on $T_0<T_{1}$.}\\
\int_{0^+} {y \,f(y)} m(dy) < + \infty\quad&\Longrightarrow \quad  \int_{0}^{T_{0}} f(X_{s}) ds < + \infty\quad\hbox{on $T_0<T_{1}$.}
\end{align*}
This concludes the proof of the direct implications in Theorem~\ref{Integ-Ext-bounded}.

\me Now, assume for instance that $\int_{0}^{T_{0}} f(X_{s})ds = + \infty$ on $T_0<T_{1}$. Then, \textit{a fortiori}, $\int_{0}^{T_{0}} f(X_{s})ds = + \infty$ on $T_0<T_{1-\varepsilon}$ for any $\varepsilon\in(0,1-x)$. This implies that $\int_{0}^{T_{0}} f(X^\varepsilon_{s})ds = + \infty$ on $T_0<T_{1-\varepsilon}$. But $T_0<T_{1-\varepsilon}$ happens with probability $x/(1-\varepsilon)>0$ by definition of  the natural scale. We deduce from Theorem~\ref{Integ-Ext} that $\int_{0^+} {y \,f(y)} m(dy) < + \infty$ does not hold and hence, because $f$ is non-negative,  that $\int_{0^+} {y \,f(y)} m(dy) = + \infty$. This provides the first $\Leftarrow$ implication in Theorem~\ref{Integ-Ext-bounded}. The second $\Leftarrow$ implication in Theorem~\ref{Integ-Ext-bounded} is proved using similar arguments.

\me The result at boundary $b$ is proved similarly.
\end{proof}

\bi
Let us illustrate Theorem~\ref{Integ-Ext-bounded} by simple examples from population genetics, which will be used as  central arguments in Section \ref{camille}.

\bi
\begin{example}
\label{cor:wright-fisher} The neutral Wright-Fisher  diffusion.
Let $X$ be the stochastic process solution of 
$$dX_{t} = \sqrt{X_{t}(1-X_{t})} \,dB_{t}\ ;\ X_{0}\in (0,1).$$
 The process $X$  is on natural scale on $(0,1)$ with speed measure $m(dy)=\frac{dy}{y(1-y)}$. Since
$\int_{0+} y\,m(dy)<+\infty\ \hbox{and}\  \int^{1-} (1-y)\,m(dy)<+\infty,
$
it  reaches $0$ or $1$ in finite time a.s..

\me Now, setting $\,f(y)=1/(1-y)$, we have
$\,
\int^{1-} (1-y)f(y)\,m(dy)=+\infty
\,$
and  Theorem~\ref{Integ-Ext-bounded}  yields
\begin{align*}
\mathbb{P}_{x}\Big(\{\int_{0}^{T_{1}}  \frac{1}{1-X_{s}} \,ds = + \infty \} \cap\{T_1< T_0\}\Big)  = \mathbb{P}_{x}\big(T_1< T_0\big) 
\end{align*}
for any $x\in (0,1)$.

\me Therefore, since $\{T_1=+\infty\}=\{T_0<T_1\}$ and $1/(1-X_{t})=1$ for all $t\geq T_0$,
\begin{align}
\label{integ-fix}
\mathbb{P}_{x}\big(\int_{0}^{T_{1}}  \frac{1}{1-X_{s}} \,ds = + \infty\big) = 1.
\end{align}
\end{example}

\bi
\begin{example}
\label{cor:wright-fisher-selection}
The Wright-Fisher diffusion with selection.   Let  $Y$ be the process solution of $$dY_{t} = \sqrt{Y_{t}(1-Y_{t})} \,dB_{t} + r\,Y_{t}\,(1-Y_{t}) dt ;\ Y_{0}\in (0,1).$$
Its scale function on $(0,1)$ is given by
$\,s(x) = \frac{1}{2r}(1 -  e^{-2rx})\,$ and its
  speed measure is $\,m(dy)=\frac{2dy}{  e^{-2ry}\,y(1-y)}$.
  In particular, we deduce that
$\,
\int_{0+} s(y)\,m(dy)<+\infty\quad\hbox{and}\quad \int^{1-} (s(1)-s(y))\,m(dy)<\infty.$ Hence, the process reaches $0$ or $1$ in finite time a.s.. 
Setting as in the previous example $\,f(y)= 1/(1-y)$, we have
$\,\int^{1-} (s(1)-s(y))f(y)\,m(dy)=+\infty\,$
and  Theorem~\ref{Integ-Ext-bounded} yields
\begin{align}
\label{integ-fix-sel}
\mathbb{P}_{x}\big(\int_{0}^{T_{1}}  \frac{1}{1-Y_{s}} \,ds = + \infty  \big)  = 1. 
\end{align}
 \end{example}

\me
\subsection{Extension to non-homogeneous processes by use of Girsanov transform}

\me We are interested in 
generalized  one-dimensional stochastic differential equations of the form
\be
\label{eq:girsanov} dX_{t} = \sigma(X_{t}) dB_{t}  + b(X_t)dt + q(X_{t}, \theta_{t}) dt, X_{0}>0, \ee
where $(B_{t}, t\ge 0)$ is a Brownian motion for some  filtration $({\cal F}_{t})_{t}$ and $(\theta_{t}, t\geq 0)$ is predictable with respect to $({\cal F}_{t})_{t}$. The process $(\theta_{t})_{t}$ can for example model an environmental heterogeneity. Other examples will be given in Section \ref{WPpop}.

\bi Assumption $(H)$:  {\it We consider real functions $\sigma$ and $b$ such that for any Brownian motion $W$ on some probability space, the  one-dimensional stochastic differential equation $dZ_{t} = \sigma(Z_{t}) dW_{t}  + b(Z_t)dt, Z_{0}>0$ satisfies the assumptions of Corollary~\ref{cor:to}. 
}

\bi \begin{theorem}
\label{girsanov}  
Let us consider a solution $X$ of  \eqref{eq:girsanov} where $\sigma$ and $b$ satisfy Assumption $(H)$. We also assume that $T_{0}=T_{0}^X <+\infty$ almost surely and that the  sequence $(T_{k}^X)_{k\in \mathbb{N^*}}$ tends almost surely to infinity as $k$ tends to infinity. 

\me 
Next, we assume that for any $k\in \mathbb{N}^*$,
\begin{align}
\label{cond-girsanov}
\mathbb{E}\Big(\exp\Big(\frac{1}{2} \int_{0}^{T_{k}^X} \frac{q^2(X_{s},\theta_{s})}{\sigma^2(X_{s})}\,ds\Big)\Big) <+\infty.
\end{align}

\me Let $f$ be a non negative locally bounded measurable function on $(0,+\infty)$. We have
\begin{align*}
\int_{0^+} f(y)s(y)\,m(dy) = + \infty\quad
&\Longleftrightarrow\quad \int_{0}^{T_{0}^X} f(X_{s}) ds = + \infty \quad \hbox{almost surely},\\
\int_{0^+} f(y)s(y)\,m(dy) < + \infty\quad
&\Longleftrightarrow\quad \int_{0}^{T_{0}^X} f(X_{s}) ds < + \infty \quad \hbox{almost surely},
\end{align*}
where $s$ and $m$ are defined in~\eqref{scalespeed}.
\end{theorem}

\me Note that~\eqref{cond-girsanov} holds true as soon as, for all $k\in\R_+$, 
\be
\label{cond:girsanov}\sup_{x\in (0,k),\theta}|q(x,\theta)/\sigma(x)|<+\infty.\ee


\me \begin{proof}
We use the Girsanov Theorem, as stated for example in Revuz-Yor \cite{Revuz} Chapter 8 Proposition 1.3.


\me Let us consider the diffusion process $X^k$ on $[0,k]$, absorbed when it reaches $0$ or $k$, at time $\tau_k:=T_0^X\wedge T_k^X$. 
\me The exponential martingale $\,{\cal E}(L^k)_{t}$, where
$\,L^k_{t} = - \int_{0}^{t\wedge \tau_{k}} \frac{q(X_{s},\theta_{s})}{\sigma(X_{s})}\,\,dB_{s}$,  is uniformly integrable thanks to \eqref{cond-girsanov} and Novikov's criterion.
Define for any $x>0$ the probability $\mathbb{Q}_{x}$ with  $\,\frac{d\mathbb{Q}_{x}}{d\mathbb{P}_{x} }|_{{\cal F}_{t}}= {\cal E}(L)_{t}$. Then, the process $\omega = B - \langle B, L\rangle$ is a $\mathbb{Q}_{x}$-Brownian motion and, under $\mathbb{Q}_x$, $X$ is solution to the SDE
\begin{align*}
dX_t=\sigma(X_t)d\omega_t+b(X_t)dt.
\end{align*}
Hence $s$ restricted to $(0,k)$ is the scale function of $X^k$ under $\mathbb{Q}_x$. Since $s$ and $f$ are both bounded in a vicinity of $k$, we deduce from Theorem~\ref{Integ-Ext-bounded} that
\begin{align*}
\int_0^{\tau_k} f(X_t)dt<+\infty\text{ a.s., under }\mathbb{Q}_x(\cdot\mid T_k^X< T_0^X).
\end{align*}

\me Note also that, since we assumed that $T_k$ tends almost surely to infinity, we have up to a $\mathbb{P}_x$-negligible event,
\begin{align*}
\left\{\int_0^{T_0} f(X_t)\,dt=+\infty\right\}=\bigcup_{k=0}^{+\infty} \left\{\int_0^{\tau_k} f(X_t)\,dt=+\infty\right\}
\end{align*}
and hence
\begin{align*}
\mathbb{P}_x\Big(\int_0^{T_0} f(X_t)\,dt=+\infty\Big)=\lim_{k\rightarrow+\infty}\mathbb{P}_x\Big(\int_0^{\tau_k} f(X_t)\,dt=+\infty\Big).
\end{align*}
 But, by definition of $\mathbb{Q}_x$ and by Theorem~\ref{Integ-Ext-bounded}, we have
 \begin{align}\label{raisonnement}
\mathbb{P}_{x} \Big( \int_{0}^{\tau_k} f(X_{t}) dt =+\infty\Big) 
&=  \mathbb{E}^{\mathbb{Q}_{x}} \Big(\un_{\int_{0}^{\tau_{k}} f(X_{t}) dt = +\infty} \,  {\cal E}\Big(\int_{0}^{\tau_{k}}  \frac{q(\omega_{s},\theta_{s})}{\sigma(\omega_{s})}\,d\omega_{s}  \Big) \Big)\\
&=
\begin{cases}
0\text{ if }\int_{0+} s(y)f(y)\,m(dy)<+\infty\\
\mathbb{E}^{\mathbb{Q}_{x}} \Big(\un_{T_0<T_k} \,  {\cal E}\left(\int_{0}^{\tau_{k}}  \frac{q(\omega_{s},\theta_{s})}{\sigma(\omega_{s})}\,d\omega_{s}  \right) \Big)\text{ otherwise}
\end{cases}
\\
&=\begin{cases}
0\text{ if }\int_{0+} s(y)f(y)\,m(dy)<+\infty\\
\mathbb{P}_{x}(T_0<T_k)\text{ otherwise.}
\end{cases}
\end{align}
Letting $k$ tend to infinity concludes the proof.
\end{proof}

\section{Applications to population genetics}
\label{camille}

\subsection{Wright Fisher equation with variable population size}
\label{WPpop}

We are interested in the allelic fixation in a population with    variable size that goes almost surely to extinction. The main question is whether  one allele has time or not to get fixed before the population goes extinct.  We will see that it depends on the behavior of the diffusion coefficient (near extinction) in the equation satisfied by the population size.

\subsubsection{Probability of fixation before extinction - Neutral case}

\bigskip
Consider the process $(N_t,X_t)_{t\geq 0}$ solution to the system of stochastic differential equations
\begin{align}
\begin{cases}
\label{eq:syst-N-X} 
dN_t&=\sigma(N_t)\,dB_t,\ N_0>0,\\
dX_t&=\sqrt{\frac{X_t(1-X_t)}{f(N_t)}}\,dW_t,
\end{cases}
\end{align}
where $B,W$ are independent  one-dimensional Brownian motions and $\sigma,f:(0,+\infty)\rightarrow(0,+\infty)$ are locally H\"older functions. 
The system is well defined for all time  $t<T_0^N=\inf\{t\geq 0,\ N_{t-}=0\}$, which is called the extinction time of the system. We set $N_t=0$ and $X_t=\d$ for all $t\geq T_0^N$, where $\d\notin [0,1]$ is a cemetery point.
The stochastic process 
$(N_t,t\geq0)$ models the population size dynamics, while $(X_t,t\geq0)$
represents the dynamics of the proportion of a given allele, or type, in the population. We also denote by $T_F=\inf\{t\geq 0,\ X_t=0\text{ or }1\} = T_0^X\wedge T_1^X$. We say that fixation occurs before extinction if and only if $T_F<T_0^N$ (otherwise we have $T_F=+\infty$). The following result provides a necessary and sufficient criterion ensuring this event to happen with probability one.

\begin{theorem}
\label{thm:ext-fix}
Fixation occurs before extinction with probability one if and only if
\begin{align}
\label{fix-ext}
\int_{0+}\frac{y}{\sigma^2(y)f(y)}\,dy=+\infty.
\end{align}
\end{theorem}

\begin{remark}
Note that the counterpart of the above result is: $\P(T_0^N<T_F=+\infty)>0$ if and only if $\int_{0+}\frac{y}{\sigma^2(y)f(y)}\,dy<\infty$. 
\end{remark}

\begin{remark}
Note that  $f(0)$ can be null or not. Nevertheless  the case $f(0)=0$  is  more interesting  and biologically motivated (see \cite{Coron}). An example will be studied in Proposition \ref{Prop2dim}.
\end{remark}
\begin{proof}
We define the random number
\begin{align*}
T_{max}=\int_0^{T_0^N} \frac{1}{f(N_s)}\,ds
\end{align*}
and the time change $\tau(t)$, for all $t\in[0,T_{max}]$, as the unique positive real number satisfying
\begin{align}
\label{timechange}
\int_0^{\tau(t)} \frac{1}{f(N_s)}\,ds=t.
\end{align}
In particular, $\tau$ is increasing and $\,T_0^N=\tau(T_{max})$.
Now, we set for $ t< T_{max}$ 
\begin{align*}
\hat{X}_t=X_{\tau(t)}. \end{align*}
 The time change formula implies that $(\hat{X}_{t}, t<T_{max})$ is solution to the stochastic differential equation
\begin{align*}
d\hat{X}_t=\sqrt{\hat{X}_t(1-\hat{X}_t)}\,d\tilde{W}_t,\ \hat{X}_0=X_0,\ t\in[0,T_{max})
\end{align*}
where $\tilde{W}$ is a standard Brownian motion.  We denote by $\hat{T}_F = \inf\{t>0, \hat{X}_t \in \{0,1\}\}$  the (possibly infinite) absorption time of $\hat{X}$. 

\bi
(i) {Assume that $\int_{0+}\frac{y}{\sigma^2(y)f(y)}\,dy=+\infty$.} In this case, $T_{max}=+\infty$ by Theorem~\ref{Integ-Ext} and $\hat{X}$ reaches $0$ or $1$ in finite time almost surely.  In particular, $\tau(\hat{T}_F)<T_0^N$ almost surely and
\begin{align*}
X_{\tau(\hat{T}_F)}=\hat{X}_{\hat{T}_F}=0\text{ or }1.
\end{align*}
As a consequence, $T_F= \tau(\hat{T}_F) <T_0^N$ with probability one, and hence fixation occurs before extinction almost surely.

\bi
(ii) {Assume that $\int_{0+}\frac{y}{\sigma^2(y)f(y)}\,dy<+\infty$.}  In this case $T_{max}<+\infty$ with probability one by Theorem~\ref{Integ-Ext}. 

Let $\tilde{W}'$ be a Brownian motion independent from $B$ and consider $\hat{X}'$ the solution to the SDE $d\hat{X}'_t=\sqrt{\hat{X}'_t(1-\hat{X}'_t)}\,d\tilde{W}'_t,\ \hat{X}'_0=X_0$. We define  for $t<T_0^N$  the time changed $X'_t=\hat{X}'_{\tau^{-1}(t)}$, so that $(N,X')$ is solution to the SDE system~\eqref{eq:syst-N-X} and hence, by uniqueness in law of the solution to this system, $(N,X')$ and $(N,X)$ have the same law. Since $(N,\hat{X}')$ and $(N,\hat{X})$ can be obtained as the same  function of $(N,X')$ and $(N,X)$ respectively, we deduce that they share the same law up to time $T_{max}$. Then we have
\begin{align*} &\P(X_t \in (0,1) \,\, \forall t<T_0^N \hbox{ and } X_{T_0^N-} \hbox{ exists in } (0,1))\nonumber\\
&\quad=\P(\hat X_t \in (0,1) \,\,\forall t<T_{max} \hbox{ and } \hat X_{T_{max-}} \hbox{ exists in } (0,1))\nonumber \\
 &\quad= \P(\hat X'_t \in (0,1) \,\,\forall t< T_{max} \hbox{ and }\hat X'_{T_{max-}} \hbox{ exists in } (0,1))\\
 &\quad >0,
 \end{align*}
 since $N$ and $\hat{X}'$ are independent and $\hat{X}'$ is a Wright-Fisher diffusion.
This concludes the proof, since  $\{X_t \in (0,1), \forall t<T_0^N \hbox{ and } X_{T_0^N-} \hbox{ exists in }(0,1)\} $ $\subset \{T_0^N<T_F\}$, therefore
$\P(T_0^N<T_F)>0$.
\end{proof}

\me Theorem \ref{thm:ext-fix} can be extended in a natural way to the case where $N$ is not on natural scale. We  consider the following classical  modeling of Wright-Fisher diffusion with variable population size, which corresponds to  $f(y)=y, \forall y\in \mathbb{R}_{+}$. 

 \begin{proposition}\label{Prop2dim} Let us consider the two-dimensional stochastic system 
 \begin{align*}
 \begin{cases}
dN_t&= \sigma(N_{t})\,dB_{t} + N_{t}(\beta- c N_{t})dt,\ N_0>0, \,c\geq 0;\,\\
dX_t&=\sqrt{\frac{X_t(1-X_t)}{N_t}}\,dW_t,\ X_0>0,
\end{cases}
\end{align*}
with two independent Brownian motions $B$ and $W$. 
\begin{itemize}
\item[(i)] Fixation occurs before extinction with probability one if and only if
\begin{align*}
\int_{0+}\frac{1}{\sigma^2(y)}\,dy=+\infty.
\end{align*}
\item[(ii)] Under this condition,
\begin{align}
\label{integ-var}
{\mathbb{P}_{x}\Big(\int_{0}^{T_{1}^X\wedge T_0^N}  \frac{1}{N_s(1-X_{s})} \,ds = + \infty \Big)  = 1. }
\end{align}
\end{itemize}
\end{proposition}

\begin{proof} (i)
The extension of Theorem \ref{thm:ext-fix} to $N$  with general scale   function $s$ is immediate, using Corollary \ref{cor:to}. Condition \eqref{fix-ext} becomes 
\begin{align}
\int_{0+}\frac{s(y)}{s'(y)\,\sigma^2(y)f(y)}\,dy=+\infty.
\end{align}
Using \eqref{scalespeed} we note that for the present case, $s(y) \sim_{y\to 0} y\,s'(y)$, which allows to conclude. 

(ii) Using notations of Theorem \ref{thm:ext-fix}, we get $T_{max}=+\infty$ a.s., $\hat X_t=X_{\tau(t)}$ for all $t>0$, and $(\hat X_t,t\geq0)$ is a neutral Wright-Fisher diffusion process. From Equation \eqref{integ-fix} we have
\begin{align}
\mathbb{P}_{x}\Big(\int_{0}^{T_{1}^{\hat X}}  \frac{1}{1-\hat X_{s}} \,ds = + \infty\Big)  = 1.
\end{align}

\me Noting that $\tau(T_1^{\hat X})=T_1^X$, that $d\tau(t)=N_{\tau(t)}dt$, and $\tau(+\infty)=T_0^N$ we get the result.
\end{proof}

\bi The previous corollary highlights the major effect of the demography on the maintenance of  genetic diversity. The behavior of $\sigma(N)$ near extinction plays a main role. For the usual  demographic term $\sigma(N)=\sqrt{N}$,  we have almost sure  fixation before extinction, but for a small perturbation of this diffusion term, taking for example $\sigma(N)= N^{(1-\varepsilon)/2}$, $\varepsilon>0$,   extinction before fixation occurs with positive probability. 
An example of extinction before fixation is illustrated in  Figure  1 and the effect of $\varepsilon$ on the probability of  extinction before fixation is numerically studied in Figure  2

\begin{figure}
\begin{center}
\includegraphics[width=0.7\textwidth]{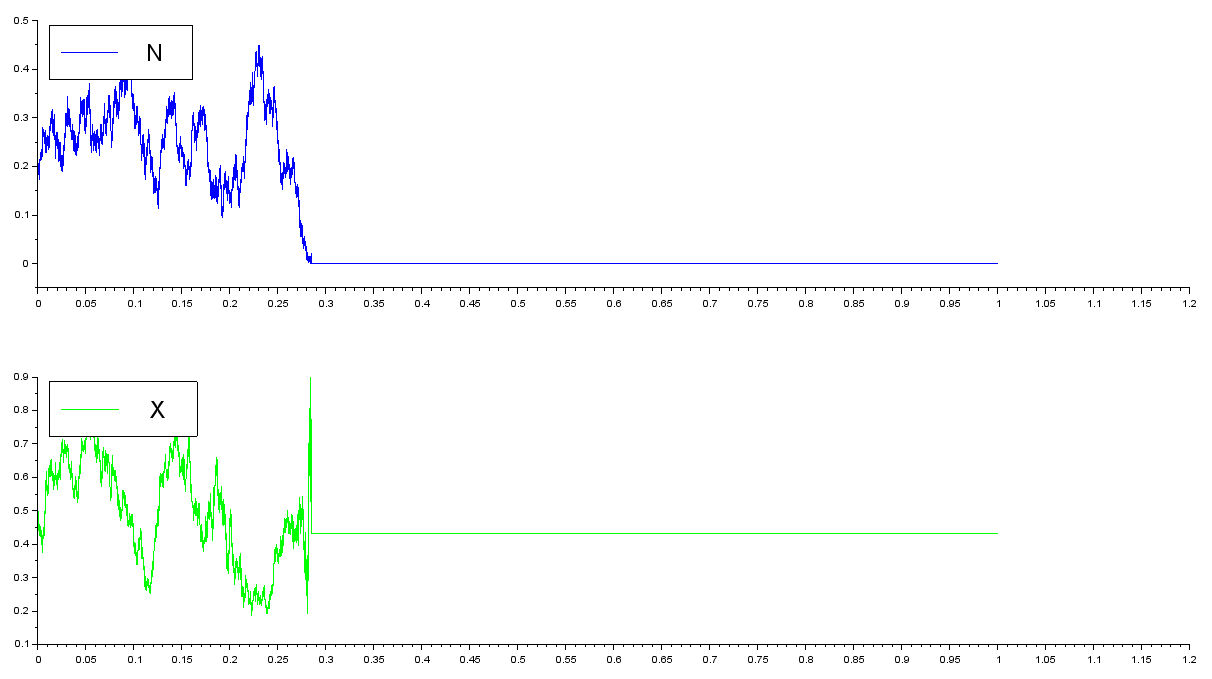}
\caption{\small We plot a trajectory of the $2$-dimensional diffusion process $(N,X)$ such that $dN_t=\sqrt{N_t^{(1-\varepsilon)}}dB^1_t+N_t(r-cN_t)dt$ and $dX_t=\sqrt{\frac{X_t(1-X_t)}{N_t}}$, with $\varepsilon=0.4$, $r=-1$ and $c=0.1$. For this trajectory, fixation does not occur before extinction.}\label{fig:sim-ext}
\end{center}
\end{figure}

\begin{figure}[h]
\label{fig:role-epsilon}
\centering
\includegraphics[width=0.7\textwidth]{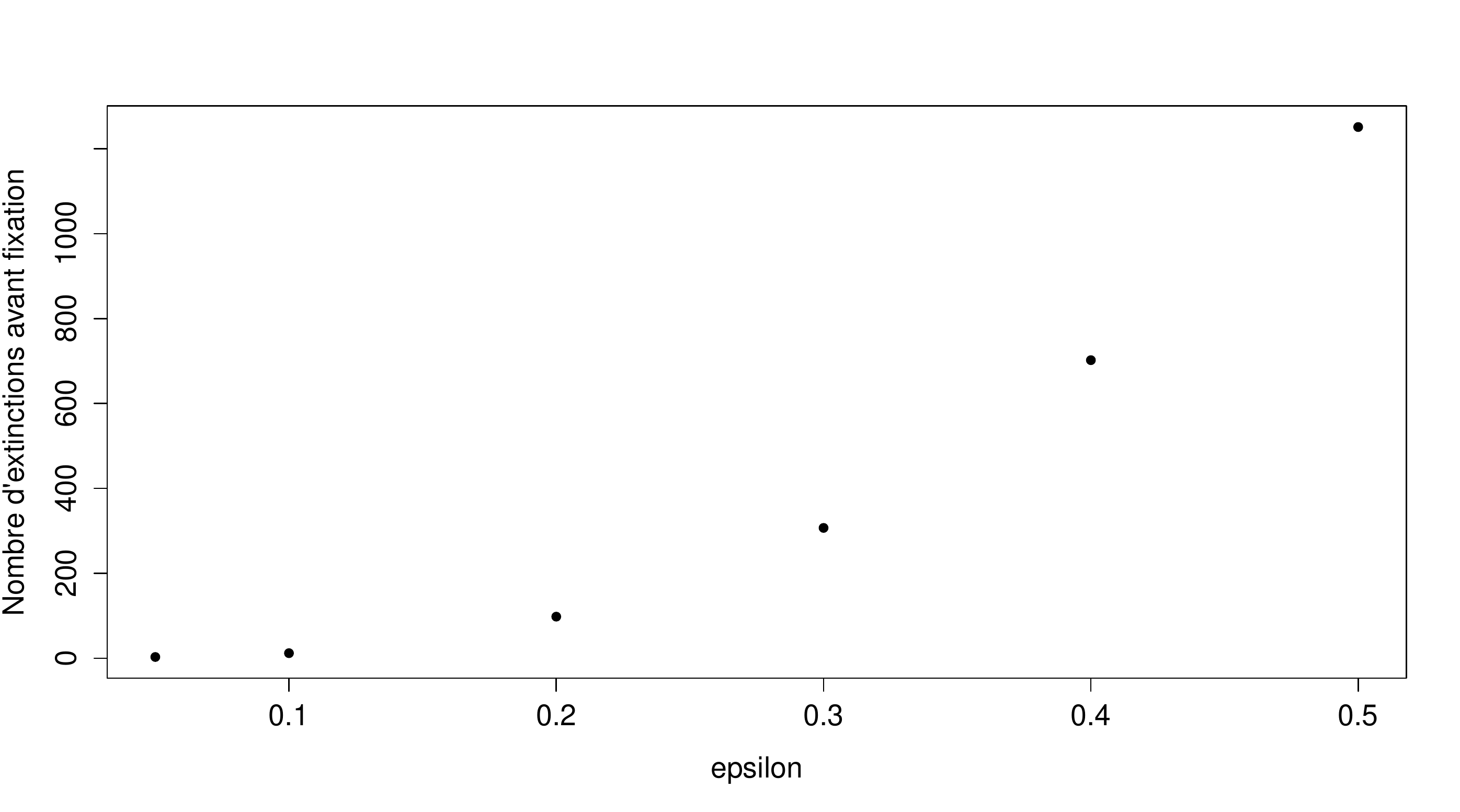}
\caption{\small For different values of $\varepsilon$, we simulate $10000$ trajectories of the $2$-dimensional diffusion process $(N,X)$ such that $dN_t=\sqrt{N_t^{(1-\varepsilon)}}dB^1_t+N_t(r-cN_t)dt$ and $dX_t=\sqrt{\frac{X_t(1-X_t)}{N_t}}$, with $r=-1$ and $c=0.1$. We plot the number of simulations for which fixation does not occur before extinction. This number is increasing with $\varepsilon$.}
\end{figure}

\subsubsection{Wright-Fisher equation with selection and variable population size}

Let us consider a $2$-types competitive Lotka-Volterra stochastic system as introduced in \cite{CattiauxMeleard}:

\ba
\begin{cases}
dN_1(t)&=\sqrt{N_1(t)}dW^1(t)+N_1(t)(r_1-c(N_1(t)+N_2(t)))dt\\
dN_2(t)&=\sqrt{N_2(t)}dW^2(t)+N_2(t)(r_2-c(N_1(t)+N_2(t)))dt,
\end{cases}
\ea
where $(B^1,B^2)$ is a standard Brownian motion, $c>0$. Setting $N(t)=N_1(t)+N_2(t)$ and $X(t)=N_1(t)/N(t)$ leads to
\ba
\begin{cases}
dN_{t}& = \sqrt{N_{t}}\,dB^1_{t} + N_{t}(r_1- c N_{t})dt + (r_2-r_1) N_{t}(1-X_{t})dt;\\
dX_{t}& =\sqrt{\frac{X_{t}(1-X_{t})}{N_{t}}}\,dB^2_{t} - (r_2-r_1) X_{t}(1-X_{t})dt,
\end{cases}
\ea
with two independent Brownian motions $B^1$ and $B^2$. The parameter $r_2-r_1$ represents the selective advantage of the type of population $2$.
 
\begin{proposition}
\begin{itemize}
\item[(i)] Fixation occurs before extinction almost surely.
\item[(ii)] We get \begin{align*}
\mathbb{P}_{x}\bigg(\int_{0}^{T_{1}^X\wedge T_0^N}  \frac{1}{N_s(1-X_{s})} \,ds = + \infty \bigg)  = 1.
\end{align*}
\end{itemize}

\end{proposition}

\begin{proof} We use a $2$-dimensional Girsanov theorem: let us consider the exponential martingale $\mathcal{E}(L^k)_t$, where $$L^k_t=-(r_2-r_1)\left(\int_0^{t\wedge T_k^N}(1-X_s)\sqrt{N_s}dB^1_s-\int_0^{t\wedge T_k^N}\sqrt{N_sX_s(1-X_s)}dB^2_s\right).$$ For each $k$, the martingale $\mathcal{E}(L^k)_t$ is uniformly integrable. Under the probability $\mathbb{Q}$ such that $\frac{d\mathbb{Q}}{d\mathbb{P}}|_{{\cal F}_{t}}=\mathcal{E}(L^k)_t$, the process $$w=(B-\langle B,L\rangle,W-\langle W,L\rangle)$$ is a bi-dimensional Brownian motion, and the process $(N,X)$ is solution to the stochastic differential equation

\begin{align}\label{eqdim2}
\begin{cases}
dN_{t}& = \sqrt{N_{t}}\,dB^1_{t} + N_{t}(r_1- c N_{t})dt\\
dX_{t}& =\sqrt{\frac{X_{t}(1-X_{t})}{N_{t}}}\,dB^2_{t}.
\end{cases}
\end{align}
stopped at $T^N_k$. Let us first prove $(i)$.
\begin{align*}
\lim_{k\to+\infty}\mathbb{P}(T_{0,1}^X<T_0^N)&=\lim_{k\to+\infty}\mathbb{P}(T_{0,1}^X<T_0^N,T_0^N<T_k^N)+\mathbb{P}(T_{0,1}^X<T_0^N,T_0^N>T_k^N)\\
&\leq \lim_{k\to+\infty}\mathbb{P}(T_{0,1}^X<T_0^N,T_0^N<T_k^N)+\mathbb{P}(T_0^N>T_k^N)\\&=\lim_{k\to+\infty}\mathbb{P}(T_{0,1}^X<T_0^N,T_0^N<T_k^N),
\end{align*}
where the last equality comes from the stochastic domination of the stochastic process $N$ by a logistic diffusion, which satisfies \eqref{extinction}. Therefore
\begin{align*} \lim_{k\to+\infty}\mathbb{P}(T_{0,1}^X<T_0^N)&=\lim_{k\to+\infty}\mathbb{E}^{\mathbb{Q}}(\un_{T_{0,1}^X<T_0^N,T_0^N<T_k^N}\,\mathcal{E}(L^k_{T_k^N}))\\&=\lim_{k\to+\infty}\mathbb{E}^{\mathbb{Q}}(\un_{T_0^N<T_k^N}\,\mathcal{E}(L^k_{T_k}))\quad\quad \text{from Proposition \ref{Prop2dim}}\\&
=\lim_{k\to+\infty}\mathbb{P}(T_0^N<T_k^N)
\end{align*}
Now for $(ii)$,
\begin{align*}
\mathbb{P}_{x}&\Big(\int_{0}^{T_{1}^X\wedge T_0^N}  {1\over N_s(1-X_{s})} \,ds = + \infty \Big) \\&=\lim_{k\to+\infty}\mathbb{P}_{x}\Big(\{\int_{0}^{T_{1}^X\wedge T_0^N\wedge T_k}  {1\over N_s(1-X_{s})} \,ds = + \infty\}\cap \{T_1^X\wedge T_0^N<T_k^N\} \Big)\\&= \lim_{k\to+\infty}\mathbb{E}^{\mathbb{Q}_x}\Big(\un_{T^X_1\wedge T^N_0<T^N_k}\un_{\int_{0}^{T_{1}^X\wedge T_0^N\wedge T_k}{1\over N_s(1-X_{s})} \,ds = + \infty}\mathcal{E}(L^k_{T_k})\Big) \\&=\mathbb{P}_x(T^X_1\wedge T^N_0<T^N_k),
\end{align*}
since, from Proposition \ref{Prop2dim}, $\int_{0}^{T_{1}^X\wedge T_0^N\wedge T_k}{1\over N_s(1-X_{s})} \,ds = + \infty\,$ a.s. under $\,\mathbb{Q}_x$.
\end{proof}

\subsection{Successive fixations for the multi-allelic Wright-Fisher case}

We consider now a neutral $L$-type Wright-Fisher diffusion $\,X=(X^1,\cdots, X^L)\,$ with $L$ types (Ethier-Kurtz \cite{EthierKurtz}, pp. $435-439$) describing the dynamics of the respective  proportions of the $L$ alleles. We are interested in the study of the successive extinctions of alleles. 

\me Since $\,X^1+\cdots+ X^L=1$, it is enough to study the dynamics of the process $\,(X^1,\cdots, X^{L-1} )$. We know (see for example  \cite[Chap. 10]{EthierKurtz}) that this diffusion admits the following infinitesimal generator.

\ban
\mathcal{L}_1f(p_1,\cdots,p_{L-1})&=\sum_{i=1}^{L-1} \,p_i(1-p_i)\,\frac{\partial^2f}{\partial p_i^2}(p_1,\cdots,p_{L-1})\\
&\hskip 0.5cm -\sum_{i\neq j\in[\![1,L-1]\!]} \,p_ip_j\,\frac{\partial^2f}{\partial p_i\partial p_j}(p_1,\cdots,p_{L-1}). \label{WFLtypes}
\ean 
We can represent the diffusion $(X^1_{t},\cdots, X^{L-1}_{t})_{t\geq0}$ in the following way: let us start by distinguishing $2$ types of alleles, the allele $1$ and the others. The stochastic process $(X^1(t))_{t\geq0}$ is a neutral Wright-Fisher diffusion and writes
\ba dX^1(t)&=\sqrt{X^1(t)(1-X^1(t))}\,dB^1_t,\ea where $(B^1_t,t\geq0)$ is a  Brownian motion.
Next, among the set of alleles that are not allele $1$ (this population has size $(1-X^1(t))$ at time $t$),we can again distinguish $2$ types of alleles, the allele $2$ and the others. The proportion of allele $3$ in this new population satisfies:
$$d\left(\frac{X^2(t)}{1-X^1(t)}\right)=\sqrt{\frac{ \frac{X^2(t)}{1-X^1(t)}\left(1-\frac{X^2(t)}{1-X^1(t)}\right)}{(1-X^1(t))}}\,dB^2_t$$ where $(B^2_t,t\geq0)$ is a Brownian motion independent  from $B^1$. Finally, the diffusion process $(X^1_{t},\cdots, X^{L-1}_{t})_{t\geq0}$ satisfies the diffusion equation:
\ban\label{diffmulti} 
d\left(\frac{X^i(t)}{1-X^1(t)-...-X^{i-1}(t)}\right)&=\sqrt{\frac{ X^i(t)(1-X^1-...-X^i(t))}{(1-X^1(t)-...-X^{i-1}(t))^3}}\,dB^i_t\ean
for all $i\in[\![1,L-1]\!]$, where $(B^1_t,...,B^{L-1}_t)_{t\geq0}$ is a $L-1$-dimensional Brownian motion. To check these assertions, it suffices to prove that the diffusion process $(X^1_{t},\cdots, X^{L-1}_{t})_{t\geq0}$ that satisfies Equation \eqref{diffmulti} admits the following quadratic variation terms: for all $i\neq j\in[\![2,L]\!]$,
\ba \begin{cases}d\langle X^i(t),X^i(t)\rangle &= X^i(t)(1-X^i(t))\,dt,\\d\langle X^i(t),X^j(t)\rangle &= - X^i(t)X^j(t)\,dt.\end{cases} \ea These results are easily obtained by using a recursion on $i$.

\bi
Our aim is to prove the following theorem:

\begin{theorem}
\label{longtime}
\begin{description}
\item[$(i)$] One of the $L$ alleles is fixed almost surely in finite time, i.e. the random variable $\,\max_{i\in \{1,\cdots, L\}} X^ i\,$ attains $1$ in finite time almost surely. 
\item[$(ii)$] Till that time, the population experiences successive (and non simultaneous) allele extinctions.
\end{description}
\end{theorem}

\me The proof of this theorem relies on the following lemma
\begin{lemme}\label{LemmeWF} Let $(X^1(t),...,X^{L-1}(t))_{t\geq0}$ be a $L-1$-dimensional Wright-Fisher diffusion process, let $1-X^L(t)=X^1(t)+...+X^{L-1}(t)$ for all time $t\geq0$, and define the change of time $\tau$ on $[0,+\infty)$ (from Example \ref{cor:wright-fisher}) such that $\int_0^{\tau(t)}\frac{1}{1-X_L(s)}ds=t$ for all $t\geq0$.
Now let $$(Y^1_t,Y^2_t,...,Y^{L-2}_t)_{t\geq0}=\left(\frac{X^1}{1-X^L}(\tau(t)),...,\frac{X^{L-2}}{1-X^L}(\tau(t))\right)_{t\geq0}.$$ The stochastic process $(Y^1_t,Y^2_t,...,Y^{L-2}_t)_{t\geq0}$ is a $L-2$-dimensional Wright-Fisher diffusion process.
\end{lemme}

\begin{proof}[Proof of Lemma \ref{LemmeWF}]
Let us denote by $\tilde{\mathcal{L}}$ the infinitesimal generator of the $L-1$-dimensional diffusion process $(\frac{X^1}{1-X^L}(t),\frac{X^2}{1-X^L}(t),...,\frac{X^{L-2}}{1-X^L}(t),1-X^L(t))_{t\geq0}$. From Equation \eqref{WFLtypes}, the infinitesimal generator $\mathcal{L}_1$ of the $L-1$-dimensional Wright-Fisher diffusion process $(X^1(t),X^2(t),...,X^{L-1}(t))_{t\geq0}$ satisfies for any bounded real-valued twice differentiable function $h$ on $\{(p_1,p_2,...,p_{L-1})|0\leq p_i\leq1 \,\forall i, p_1+p_2+...+p_{L-1}\leq1\}$:

\ba \mathcal{L}_1h(p_1,p_2,...,p_{L-1})&= \sum_{i=1}^{L-1}\gamma \,p_i(1-p_i)\frac{\partial^2h}{\partial p_i^2}(p_1,...,p_{L-1})\\
&\hskip 0.5cm -\sum_{i\neq j\in[\![1,L-1]\!]}\gamma \,p_ip_j\,\frac{\partial^2h}{\partial p_i\partial p_j}(p_1,...,p_{L-1}).\ea

\me Now for any bounded real-valued twice differentiable function $f$ defined on $\{(\tilde{x}_1,...,\tilde{x}_{L-2},1-x_L)|0\leq \tilde{x}_i\leq1 \forall i, \tilde{x}_1+...+\tilde{x}_{L-2}\leq1\}\times[0,1]$, we may write $$\tilde{\mathcal{L}}f(\tilde{x}_1,...,\tilde{x}_{L-2},1-x_L)=\mathcal{L}(f\circ g)(x_1,...,x_{L-1}),$$ where \ba g(x_1,...,x_{L-1})&=(g_1(x_1,...,x_{L-1}),...,g_{L-1}(x_1,...,x_{L-1}))\\&=\left(\frac{x_1}{1-x_L},...,\frac{x_{L-2}}{1-x_L}, x_1+...+x_{L-1}\right),\ea and $(\tilde{x}_1,...,\tilde{x}_{L-2},1-x_L)=g(x_1,...,x_{L-1})$.

\me Therefore, we obtain that 

\ba\tilde{\mathcal{L}}f(\tilde{x}_1,\tilde{x}_2,...,\tilde{x}_{L-2},1-x_L)&=\sum_{j=1}^{L-2}\frac{\gamma\tilde{x}_j(1-\tilde{x}_j)}{1-x_L}\frac{\partial^2f}{\partial\tilde{x}_j^2}(\tilde{x}_1,\tilde{x}_2,...,\tilde{x}_{L-2},1-x_L)\\&-\sum_{j\neq k\in[\![1,L-2]\!]}\frac{\gamma\tilde{x}_j\tilde{x}_k}{1-x_L}\frac{\partial^2f}{\partial\tilde{x}_j\partial\tilde{x}_k}(\tilde{x}_1,\tilde{x}_2,...,\tilde{x}_{L-2},1-x_L)\\&+\gamma x_L(1-x_L)\frac{\partial^2f}{\partial(1-x_L)^2}(\tilde{x}_1,\tilde{x}_2,...,\tilde{x}_{L-2},1-x_L)\ea
which gives the result since $d\tau(t)=(1-X^L(t))dt$.
\end{proof}

\begin{proof}[Proof of Theorem \ref{longtime}]
We prove both results by induction on $L$. For $(i)$, we know that the result is true for $L=2$. Now for $L$ alleles, note that the proportion of allele $1$ follows a $1$-dimensional Wright-Fisher diffusion. Therefore allele $1$ gets fixed or disappears almost surely in finite time. If allele $1$ gets fixed then one of the $L$ alleles gets fixed almost surely in finite time. If allele $1$ gets lost then from its extinction time, the population follows a $L-1$-type  Wright-Fisher diffusion, therefore one of the $L-1$ remaining alleles gets fixed almost surely in finite time, using the induction assumption.

\me We now prove $(ii)$. We have
\begin{align*}
\int_{0}^{T^L_{1}} \frac{1}{1-X^L_{s}} \,ds=+\infty.
\end{align*} from Example \ref{cor:wright-fisher}.
We define the time change $\tau(t)$, for all $t\in [0,+\infty)$, as the unique non-negative real number satisfying
\begin{align*}
\int_{0}^{\tau(t)} \frac{1}{1-X^L_{s}} ds = t.
\end{align*}
Now,  for all $1\leq i\leq L-1$, let us define the stochastic process $Y_t=(Y^1_t,\ldots,Y^{L-1}_t)_{t\geq0}$ such that
\begin{align*}
Y^i_t=\frac{X^i}{1-X^L}(\tau(t)) \quad\forall t\in[0,+\infty).
\end{align*}
From Lemma \ref{LemmeWF}, the $L-1$ stochastic process $(Y^1_t,Y^2_t,...,Y^{L-1}_t)_{t\geq0}$ is a $L-1$ dimensional Wright-Fisher diffusion process. By recurrence assumption, this diffusion process experiences $L-2$ successive and non simultaneous extinctions, at times denoted by $S^Y_1<...<S^Y_{L-2}<+\infty$. Therefore $\tau(S^Y_1)<...<\tau(S^Y_{L-2})<\tau(+\infty)=T_1^L$. Under the event $T^L_1<+\infty$, the times $\tau(S^Y_1)$, ..., $\tau(S^Y_{L-2})$ and $T_1^L$ correspond to the $L-1$ extinction times experienced by the population, which gives the result, since $\mathbb{P}(\cup_{i=1}^L\{T^i_1<+\infty\})=1$ from~ $(i)$.
\end{proof}

{\bf Acknowledgements:} {\sl This work  was partially funded by the Chair "Mod\'elisation Math\'ematique et Biodiversit\'e" of VEOLIA-Ecole Polytechni\-que-MNHN-F.X. It was also supported by a public grants as part of the "Investissement d'avenir" project, reference ANR-11-LABX-0056-LMH,-LabEx LMH, and reference ANR-10-CAMP-0151-02, Fondation Math\'ematiques Ja\-cques Hadamard, and by the Mission for
Interdisciplinarity at the Centre national de la recherche scientifique.}

\end{document}